\documentclass[11pt]{amsart}

\usepackage[all]{xypic}

\usepackage{latexsym}

\usepackage{amssymb}

\usepackage{amsfonts}

\usepackage{amscd}

\usepackage{amsmath,amsthm}

\newtheorem{lemma}{Lemma}[section]

\newtheorem{lem}[lemma]{Lemma}

\newtheorem{prop}[lemma]{Proposition}

{

}

\theoremstyle{definition}

\theoremstyle{remark}

\numberwithin{equation}{section}

\newenvironment{pf}{\noindent{\bf Proof.}}{\hfill $\square$\medskip}

\def\NN{{\mathbb N}}

\def\PP{{\mathbb P}}

\def\fol{{\bar f}}

\def\0ol{{\bar 0}}

\def\1ol{{\bar 1}}

\def\2ol{{\bar 2}}

\def\ol2{{\bar 2}}

\def\3ol{{\bar 3}}

\def\4ol{{\bar 4}}

\def\5ol{{\bar 5}}

\def\6ol{{\bar 6}}

\def\7ol{{\bar 7}}

\def\8ol{{\bar 8}}

\def\9ol{{\bar 9}}

\def\bold0{{\bf 0}}

\def\bold1{{\bf 1}}

\def\bold2{{\bf 2}}

\def\bold3{{\bf  3}}

\def\bold4{{\bf 4}}

\def\bold5{{\bf 5}}

\def\bold6{{\bf 6}}

\def\bold7{{\bf 7}}

\def\bold8{{\bf 8}}

\def\bold9{{\bf 9}}

\def\P2Skly{\PP^2_{Skly}}

\def\coker{\operatorname {coker}}

\def\ker{\operatorname {ker}}

\def\dim{\operatorname{dim}}

\def\Fdim{{\sf Fdim}}

\def\Gr{{\sf Gr}}

\def\QGr{\operatorname{\sf QGr}}

\def\ul1{\operatorname{\underline{1}}}

\def\sT{{\sf T}}

\def\dirlim{\mathop{\vtop{\baselineskip -100pt\lineskip -1pt\lineskiplimit 0pt

\setbox0\hbox{lim}\copy0\hbox to \wd0{\rightarrowfill}}}\limits}

\def\invlim{\mathop{\vtop{\baselineskip -100pt\lineskip -1pt\lineskiplimit 0pt

\setbox0\hbox{lim}\copy0\hbox to \wd0{\leftarrowfill}}}\limits}

\def\I11{{1 \kern -0.8pt \! \mbox{l}}}

\def\mumu{{\mu\kern-4.2pt\mu}}

\def\bfmu{{\mu\kern-4.2pt\mu}}

\def\2slash{\backslash \! \backslash}

\def\boxtimes{\setbox0\hbox{$\Box$}\copy0\kern-\wd0\hbox{$\times$}}

\date{}                                           

\begin{document}

\title[Graded modules over monomial algebras and path algebras]{Corrigendum to 
``An equivalence of categories for 
graded modules over monomial algebras and path algebras of quivers''
[J. Algebra,  353(1) (2012) 249-260]}

\author{Cody Holdaway and S. Paul Smith}

\address{ Department of Mathematics, Box 354350, Univ.
Washington, Seattle, WA 98195}

\email{codyh3@math.washington.edu, smith@math.washington.edu}



\begin{abstract}
Our published paper contains an incorrect statement of a result due to Artin and Zhang. 
This corrigendum gives the correct statement of their result and includes a new result that allows us to use the correct version of Artin and Zhang's Theorem to prove our main theorem. Thus the main theorem of our published paper is correct as stated but its proof must be modified.
\end{abstract}

\maketitle

\pagenumbering{arabic}

\setcounter{section}{0}

\section{The error}

\subsection{}
We retain the notation and definitions in our published paper \cite{HS}.

\subsection{}
Proposition 2.1 in \cite{HS} is stated incorrectly. It should be replaced by the following statement.

\begin{prop}
\label{prop.AZ2}
\cite[Prop 2.5]{AZ}
\label{prop.AZ}
Let $A$ and $B$ be $\NN$-graded $k$-algebras such that $\dim_k A_i < \infty$ and $\dim_k B_i < \infty$
for all $i$. Let $\phi:A \to B$ be a homomorphism of graded $k$-algebras. 
If $\ker \phi$ and $\coker\phi$ belong to $\Fdim A$, then $- \otimes_A B$ induces an equivalence of categories $$
\QGr A \to \frac{\Gr B}{\sT_A}
$$
where
$$
\sT_{A}=\{M\in \Gr B\;|\; M_A\in \Fdim A\}.
$$
\end{prop}

\subsection{}
As in our published paper, $A$ is a finitely presented connected monomial algebra and 
$kQ$ is the path algebra of its Ufnarovskii graph, $Q$. 
Proposition 3.3 in \cite{HS} proved the existence of a homomorphism $\fol:A \to kQ$ of graded $k$-algebras and
 \cite[Prop. 4.1]{HS} showed that $\ker \fol$ and $\coker\fol$ belong to $\Fdim A$. 
 Therefore
 Proposition \ref{prop.AZ}  implies that $- \otimes_A kQ$ induces an equivalence of categories 
 $$
\QGr A \to \frac{\Gr kQ}{\sT_A}.
$$

Thus to prove our main theorem, \cite[Thm. 4.2]{HS} and \cite[Thm. 1.1]{HS}, which says that  
$- \otimes_A kQ$ induces an equivalence of categories 
 $$
\QGr A \equiv \QGr kQ= \frac{\Gr kQ}{\Fdim kQ},
$$
we must prove that $\Fdim kQ=\sT_A$. We do this in Proposition \ref{prop.whew} below.

\section{Corrected proof}

\subsection{}
The algebra $A$ has a distinguished set of generators called {\it letters} and its relations are generated by a 
finite set of words in those letters. The vertices in $Q$ are certain words, the arrows in $Q$ are
also words, and the arrow corresponding to a word $w$ is labelled by the first letter of $w$. 
The details are in \cite[Sect. 3.3]{HS}.

\begin{lemma}
 \label{lemma.1path}
The arrows in $Q$ have the following properties.
 \begin{enumerate}
  \item 
  Different arrows ending at the same vertex have different labels. 
  \item 
Different arrows having the same label end at different vertices. 
\end{enumerate}
\end{lemma}
\begin{proof}
(1)
Let $a$ and $a'$ be different arrows ending at the vertex $v$. By definition, there are words $w$ and $w'$ such that $a=a_w$ and $a'=a_{w'}$, and letters $x$ and $x'$ such that $w=xv$ and $w'=x'v$. But $a \ne a'$ so $w \ne w'$ and therefore $x \ne x'$. But $a$ is labelled $x$ and $a'$ is labelled $x'$.  

(2)
This is obviously equivalent to (1). 
\end{proof}

\subsection{}
The homomorphism $\fol:A \to kQ$ is defined as follows: if $x$ is one of the letters generating $A$, then
$$
\fol(x):= \begin{cases}
	 \hbox{the sum of all arrows labelled $x$} & \text{}
	 \\
	 0  \hbox{ if there are no arrows labeled $x$.} & \text{}
	 \end{cases}
$$

\begin{lem}
Let $\fol:A \to kQ$ be the homomorphism above and write $A_n(kQ)$ for the right ideal of $kQ$ generated by $\fol(A_n)$.  For all $n \ge 0$,
$$
A_n(kQ)=kQ_{\ge n}.
$$
\end{lem}
\begin{pf}
Write $B=kQ$. 

We will prove that $A_nB_0=B_n$ for all $n \ge 0$. This is certainly true for $n=0$.

We will now show that $A_1B_0=B_1$. To prove this, let $a$ be an arrow in $Q$ that begins at vertex $u$ 
and ends at vertex $v$. Then there are letters $x$ and $y$ such that $a=a_w$ and $w=uy=xv$. 
The arrow $a$ is therefore labelled $x$ and $\fol(x)=\cdots + a + \cdots$. By Lemma \ref{lemma.1path}, $a$ is the {\it only} arrow labelled $x$ that ends at $v$; hence, if $e_v$ is the trivial path at vertex $v$, then $f(x)e_v=ae_v=a$. Hence $a \in A_1B_0$. Thus $B_1 \subset A_1B_0$. It is clear that $A_1B_0 \subset B_1$ so
this completes the proof that $A_1B_0=B_1$.

We now argue by induction on $n$. If $A_{n-1}B_0=B_{n-1}$, then 
$$
A_nB_0 = (A_1)^nB_0=A_1(A_1)^{n-1}B_0=A_1B_{n-1}=A_1B_0B_{n-1}=B_1B_{n-1}=B_n.
$$
This completes the proof that $A_nB_0=B_n$ for all $n \ge 0$. It follows that $A_nB=B_{\geq n}$.
\end{pf}

\begin{prop}
\label{prop.whew}
$\sT_A=\Fdim kQ$.
\end{prop}
\begin{pf}
Every $kQ$-module is an $A$-module so a right $kQ$-module that is the sum of its finite dimensional 
$kQ$-submodules is also the sum of its finite dimensional $A$-submodules. Therefore $\Fdim kQ \subset \sT_A$.

To prove the reverse inclusion, let $M\in \Gr kQ$ and suppose $M$ is in $\sT_A$; i.e., $M$ is the 
sum of its finite dimensional $A$-submodules. 
Let $m$ be a homogeneous element in $M$. Then $\dim_k(mA)<\infty$ so $mA_{\geq n}=0$ for $n \gg 0$. 
In particular, $mA_n=0$ so 
$$
m(kQ_{\ge n}) =  mA_n(kQ)=0.
$$
Hence $m(kQ)$ is isomorphic to a quotient of $kQ/kQ_{\ge n}$ and  therefore finite dimensional.
In particular, $m$ belongs to the sum of the finite dimensional $kQ$-submodules of $M$. 
Hence $M \in \Fdim kQ$.

Thus, $\sT_A \subset \Fdim kQ$ and the claimed equality follows.
\end{pf}

 This completes the proof of \cite[Thm. 4.2]{HS} and \cite[Thm. 1.1]{HS}.


\end{document}